\documentclass[a4paper,12pt]{article}
\usepackage{times, url}
\usepackage{graphicx}
\usepackage{amsmath,amssymb,amsthm,amsfonts,bm,cite,float}
\usepackage{amsmath}
\usepackage{amsfonts}
\usepackage{amssymb}

\newtheorem{theorem}{Theorem}[section]

\theoremstyle{definition}
\newtheorem{definition}[theorem]{Definition}

\theoremstyle{remark}

\numberwithin{equation}{section}

\usepackage[center]{caption}

\usepackage[none]{hyphenat}

\usepackage{color, colortbl}
\definecolor{Gray}{gray}{0.9}

\begin{document}
\setcounter{page}{1}

\begin{center}
{\LARGE \bf  A note on edge irregularity strength of Dandelion graph}
\vspace{6mm}

{\Large \bf H. M. Nagesh}\\[3mm]
Department of Science \& Humanities,\\ 
PES University - Electronic City Campus, \\
Hosur Road, Bangalore -- 560 100, India.  \\
e-mail: \url{sachin.nagesh6@gmail.com} \\[2mm]
\end{center}
\noindent

{\bf Abstract:} For a simple graph $G$, a vertex labeling $\phi:V(G) \rightarrow \{1, 2,\ldots,k\}$ is called $k$-labeling. The weight of an edge $xy$ in $G$, written $w_{\phi}(xy)$, is the sum of the labels of end vertices $x$ and $y$, i.e., $w_{\phi}(xy)=\phi(x)+\phi(y)$.
A vertex $k$-labeling is defined to be an edge irregular $k$-labeling of the graph $G$ if for every two different edges $e$ and $f$, $w_{\phi}(e) \neq w_{\phi}(f)$. The minimum $k$ for which the graph $G$ has an edge irregular $k$-labeling is called the edge irregularity strength of $G$, written $es(G)$. In this note, we find the exact value of edge irregularity strength of Dandelion graph when $\Delta(G) \geq \lceil \frac{|E(G)|+1}{2} \rceil$; and determine the bounds when $\Delta(G) < \lceil \frac{|E(G)|+1}{2} \rceil $.

\enlargethispage{\baselineskip} 
\section{Introduction} \label{sec:Intr}
Let $G$ be a connected, simple and undirected graph with vertex set $V(G)$ and edge set $E(G)$. By a labeling we mean any mapping that maps a set of graph elements to a set of numbers (usually positive integers), called \em labels\rm. If the domain is the vertex-set (the edge-set), then the labeling is called \emph{vertex labelings} (\emph{edge labelings}). If the domain is $V(G) \cup E(G)$, then the labeling is called \emph{total labeling}. Thus, for an edge $k$-labeling $\delta: E(G) \rightarrow \{1, 2,\ldots,k\}$ the associated weight of a vertex $x \in V(G)$ is $w_{\delta}(x)=\sum \delta(xy)$, where the sum is over all vertices $y$ adjacent to $x$.

Chartrand et al. \cite{4} defined irregular labeling for a graph $G$ as an assignment of labels from the set of natural numbers to the edges of $G$ such that the sums of the labels assigned to the edges of each vertex are different. The minimum value of the largest label of an edge over all existing irregular labelings is known as the \emph{irregularity strength} of $G$ and it is denoted by $s(G)$.  Finding the irregularity strength of a graph seems to be hard even for simple graphs \cite{4}.

Motivated by this, Baca et al. \cite{3} investigated two modifications of the irregularity strength of graphs, namely \emph{total edge irregularity strength}, denoted by $tes(G)$; and \emph{total vertex irregularity strength}, denoted by $tvs(G)$. Motivated by the work of Chartrand et al. \cite{4}, Ahmad et al. \cite{1} introduced the concept of edge irregular $k$-labelings of graphs.

A vertex labeling $\phi:V(G) \rightarrow \{1, 2,\ldots,k\}$ is called $k$-labeling. The weight of an edge $xy$ in $G$, written $w_{\phi}(xy)$, is the sum of the labels of end vertices $x$ and $y$, i.e., $w_{\phi}(xy)=\phi(x)+\phi(y)$. A vertex $k$-labeling of a graph $G$ is defined to be an \emph{edge irregular $k$-labeling} of the graph $G$ if for every two different edges $e$ and $f$, $w_{\phi}(e) \neq w_{\phi}(f)$. The minimum $k$ for which the graph $G$ has an edge irregular $k$-labeling is called the \emph{edge irregularity strength} of $G$, written $es(G)$. 

\section{Preliminary results}
The authors in \cite{1} estimated the bounds of the edge irregularity strength and then determined its exact value for several families
of graphs namely, paths, stars, double stars, and Cartesian product of two paths. Ahmad et al. \cite{2} determined the edge irregularity strength
of Toeplitz graphs. Tarawneh et al. \cite{7} determined the exact value of edge irregularity strength of corona product of graphs with paths. Tarawneh et al. \cite{8} determined the exact value of edge irregularity strength of disjoint union of graphs. Recently, Nagesh et al. \cite{6} determined the edge irregularity strength of line graph and line cut-vertex graph of comb graph.

The following theorem in \cite{1} establishes the lower bound for the edge irregularity strength of a graph $G$. 
\begin{theorem}
Let $G=(V,E)$ be a simple graph with maximum degree $\Delta(G)$. Then 
\begin{center} 
$es(G) \geq \max \{\lceil \frac{|E(G)|+1}{2} \rceil, \Delta(G)\}.$
\end{center} 
\end{theorem}

\section{Edge irregularity strength of Dandelion graph}
The authors in \cite{5} introduced the concept of Dandelion graph while studying the Wiener inverse interval problem.
\begin{definition} A \emph{star graph}, written $s_{n-1}$, is a graph on $n$ vertices, consisting of some vertex, connected to $n-1$ leaves.
\end{definition}
\begin{definition}
The \emph{Dandelion graph}, written $D(n,l)$, is a graph on $n$ vertices, consisted of a copy of the star $S_{n-1}$ and copy of a path $p_l$ on vertices $p_0,p_1,p_2,\ldots,p_{l-1}$, where $p_{0}$ is identified with a star center.
\end{definition}
Figure 1 shows an example of $D(17,8)$. 
\vspace{2mm}
\begin{center}
\includegraphics[width=10cm]{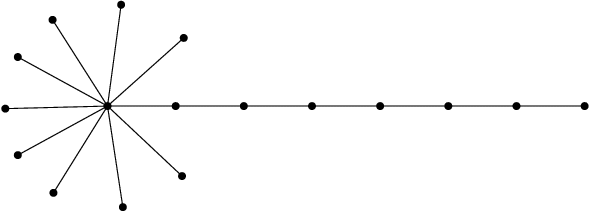}
\end{center}
\begin{center}
Figure 1
\end{center}

In the next theorem, we find the exact value of edge irregularity strength of Dandelion graph when $\Delta(G) \geq \lceil \frac{|E(G)|+1}{2} \rceil$; and determine the bounds when $\Delta(G) < \lceil \frac{|E(G)|+1}{2} \rceil $. \newpage
\begin{theorem}
Let $G=D(n,l)$ be the Dandelion graph. Then 
\begin{center}
$es(G)=n-l+1 \hspace{5mm} \text{for \ } \Delta(G) \geq \lceil \frac{|E(G)|+1}{2} \rceil$
\end{center}
For $\Delta(G) < \lceil \frac{|E(G)|+1}{2} \rceil $, \begin{center}
$\hspace{3mm}$ $\lceil \frac{n}{2} \rceil \leq es(G) \leq n-l+\lceil \frac{l}{2} \rceil$.
\end{center}
\end{theorem}
\begin{proof}
Let $G=D(n,l)$ be the Dandelion graph. Let us consider the vertex set and the edge set of $G$: 
\begin{align*} 
V(G) & =\{x_i: 1 \leq i \leq n-l\} \cup \{p_j: 0 \leq j \leq l-1\}, \\ 
E(G) & =\{p_{0}x_{i}, p_{j}p_{j+1}: 1 \leq i \leq n-l, 0 \leq j \leq l-2\}. 
\end{align*}

We consider the following three cases.

\medskip
\noindent \underline{Case 1:} When $\Delta(G) > \lceil \frac{|E(G)|+1}{2} \rceil$. Clearly, $|V(G)|=n$, $|E(G)|=n-1$, and the maximum degree $\Delta=n-l+1$. According to the Theorem 2.1, $es(G) \geq \max \{ \lceil \frac{n}{2} \rceil, n-l+1 \}$. Since $n-l+1 > \lceil \frac{n}{2} \rceil$, $es(G) \geq n-l+1$.  

To prove the equality, it suffices to prove the existence of an edge irregular $(n-l+1)-$labeling. Define a labeling on vertex set of $G$ as follows: 

Let $\phi: V(G) \rightarrow \{1,2,\ldots, n-l+1\}$ such that $\phi(x_i)=i$ for $1 \leq i \leq n-l$; $\phi(p_0)=1$; $\phi(p_1)=n-l+1$; $\phi(p_2)=4$; 
$\phi(p_{2i+1})=\phi(p_{2i})=n-l-j$ for $i=1,2,3,\ldots,$ and $j=0,1,2,\ldots$. 

The edge weights are as follows:
\begin{center}
$w_{\phi}(p_{0}x_{i})=i+1 \hspace{2mm} \text{for } 1 \leq i \leq n-l$;\\ $w_{\phi}(p_{0}p_{1})=n-l+2$,\\ $w_{\phi}(p_{1}p_{2})=n-l+5$;\\
$w_{\phi}(p_{2}p_{3})=n-l+4$;\\ $w_{\phi}(p_{i}p_{i+1})=2n-2l-j$ for $3 \leq i \leq l-2$ and $j=0,1,2,\ldots$.
\end{center}
On the basis of above calculations we see that the edge weights are distinct for all pairs of distinct edges. Therefore, 
$es(G)=n-l+1$. 

Figure 2 shows the edge irregularity strength of the Dandelion graph $D(13,5)$, where the maximum degree $\Delta=9$ and $|E(D(13,5))|=12$. 
\vspace{3mm}
\begin{center}
\includegraphics[width=9cm]{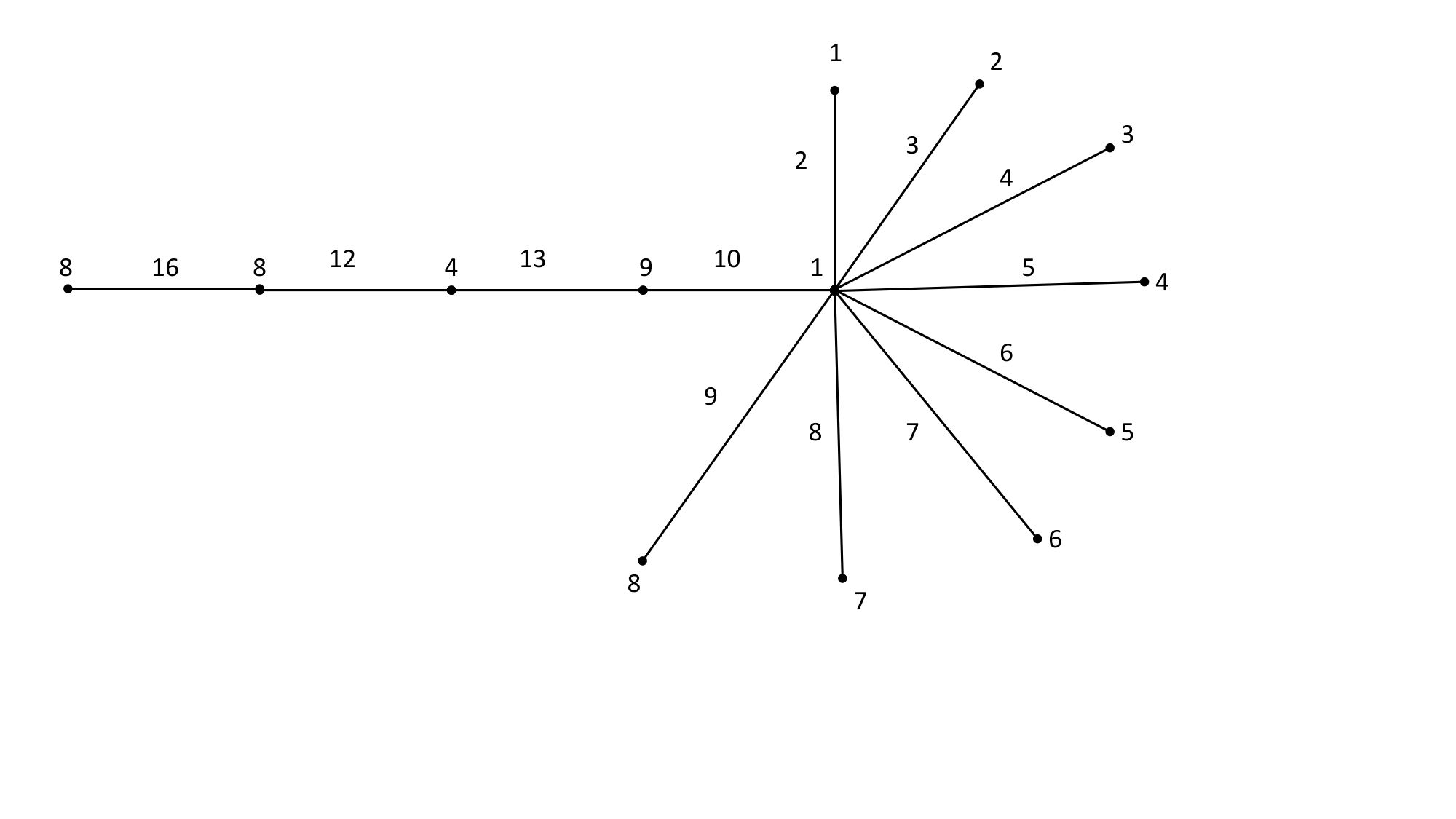}
\end{center}
\begin{center}
Figure 2
\end{center}
\noindent \underline{Case 2:} When $\Delta(G) = \lceil \frac{|E(G)|+1}{2} \rceil$. Note that $\Delta(G) = \lceil \frac{|E(G)|+1}{2} \rceil$ for the Dandelion graphs $D(2i+3,i+2)$ and $D(2i+4,i+3)$ for $i \geq1$.

Since $|E(G)|=n-1$ and the maximum degree $\Delta=n-l+1$, according to the Theorem 2.1, $es(G) \geq n-l+1$. 

To prove the equality, it suffices to prove the existence of an edge irregular $(n-l+1)-$labeling. 

Let $\phi: V(G) \rightarrow \{1,2,\ldots, n-l+1\}$ such that $\phi(x_i)=i$ for $1 \leq i \leq n-l$; $\phi(p_0)=1$; $\phi(p_1)=\phi(p_2)=n-l+1$; 
$\phi(p_{2i+1})=\phi(p_{2i+2})=n-l-j$ for $i=1,2,3,\ldots,$ and $j=0,1,2,\ldots$. 

The edge weights are as follows:
\begin{center}
$w_{\phi}(p_{0}x_{i})=i+1 \hspace{2mm} \text{for } 1 \leq i \leq n-l$;\\ $w_{\phi}(p_{0}p_{1})=n-l+2$, \\$w_{\phi}(p_{1}p_{2})=2n-2l+2$;\\
$w_{\phi}(p_{2}p_{3})=2n-2l+1$;\\ $w_{\phi}(p_{i}p_{i+1})=2n-2l-j$ for $3 \leq i \leq l-2$ and $j=0,1,2,\ldots$.
\end{center}
On the basis of above calculations we see that the edge weights are distinct for all pairs of distinct edges. Therefore, 
$es(G)=n-l+1$. 

Figure 3 shows the edge irregularity strength of the Dandelion graph $D(9,5)$, where the maximum degree $\Delta=5$ and $|E(D(9,5))|=8$. \begin{center}
\includegraphics[width=10cm]{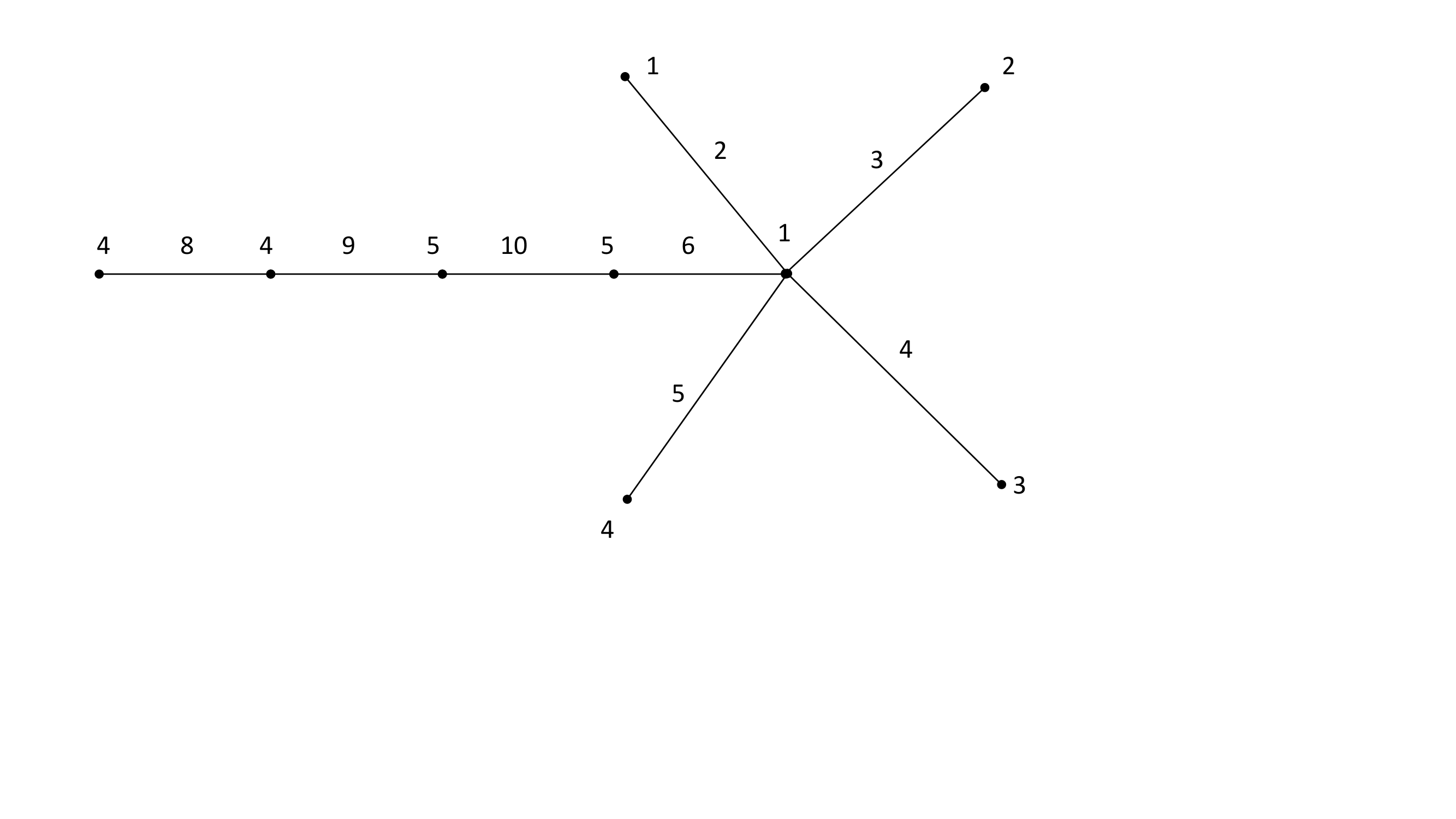}
\end{center}
\begin{center}
Figure 3
\end{center}

Finally, we have to prove that $\lceil \frac{n}{2} \rceil \leq es(G) \leq n-l+\lceil \frac{l}{2} \rceil$ for $\Delta(G) < \lceil \frac{|E(G)|+1}{2} \rceil $
\vspace{3mm}

\noindent \underline{Case 3:} When $ \Delta(G) < \lceil \frac{|E(G)|+1}{2} \rceil $. According to the Theorem 2.1, $es(G) \geq \max \{ \lceil \frac{n}{2} \rceil, n-l+1 \}$. Hence $es(G) \geq \lceil \frac{n}{2} \rceil$. For the upper bound, we define a vertex labeling $\phi$ as follows:

Let $\phi: V(G) \rightarrow \{1,2,\ldots, n-l+\lceil \frac{l}{2} \rceil\}$ such that $\phi(x_i)=i$ for $1 \leq i \leq n-l$; $\phi(p_0)=1$; $\phi(p_i)=\phi(p_{i+1})=n-l+i$ for $i=1,2,3,\ldots$. 

The edge weights are as follows:
\begin{center}
$w_{\phi}(p_{0}x_{i})=i+1 \hspace{2mm} \text{for } 1 \leq i \leq n-l$;\\ $w_{\phi}(p_{0}p_{1})=n-l+2$;\\
$w_{\phi}(p_{i}p_{i+1})=2n-2l+i+1$ for $ i=1,2,3,\ldots$.
\end{center}
On the basis of above calculations we see that the edge weights are distinct for all pairs of distinct edges. Therefore, the vertex labeling $\phi$ is an edge irregular $n-l+\lceil \frac{l}{2} \rceil-$labeling, i.e., $es(G) \leq n-l+\lceil \frac{l}{2} \rceil$ for $ \Delta(G) < \lceil \frac{|E(G)|+1}{2} \rceil $. This completes the proof.
\end{proof}
Figure 4 shows the edge irregularity strength of the Dandelion graph $D(7,5)$, where the maximum degree $\Delta=3$ and $|E(D(7,5))|=6$. \begin{center}
\includegraphics[width=10cm]{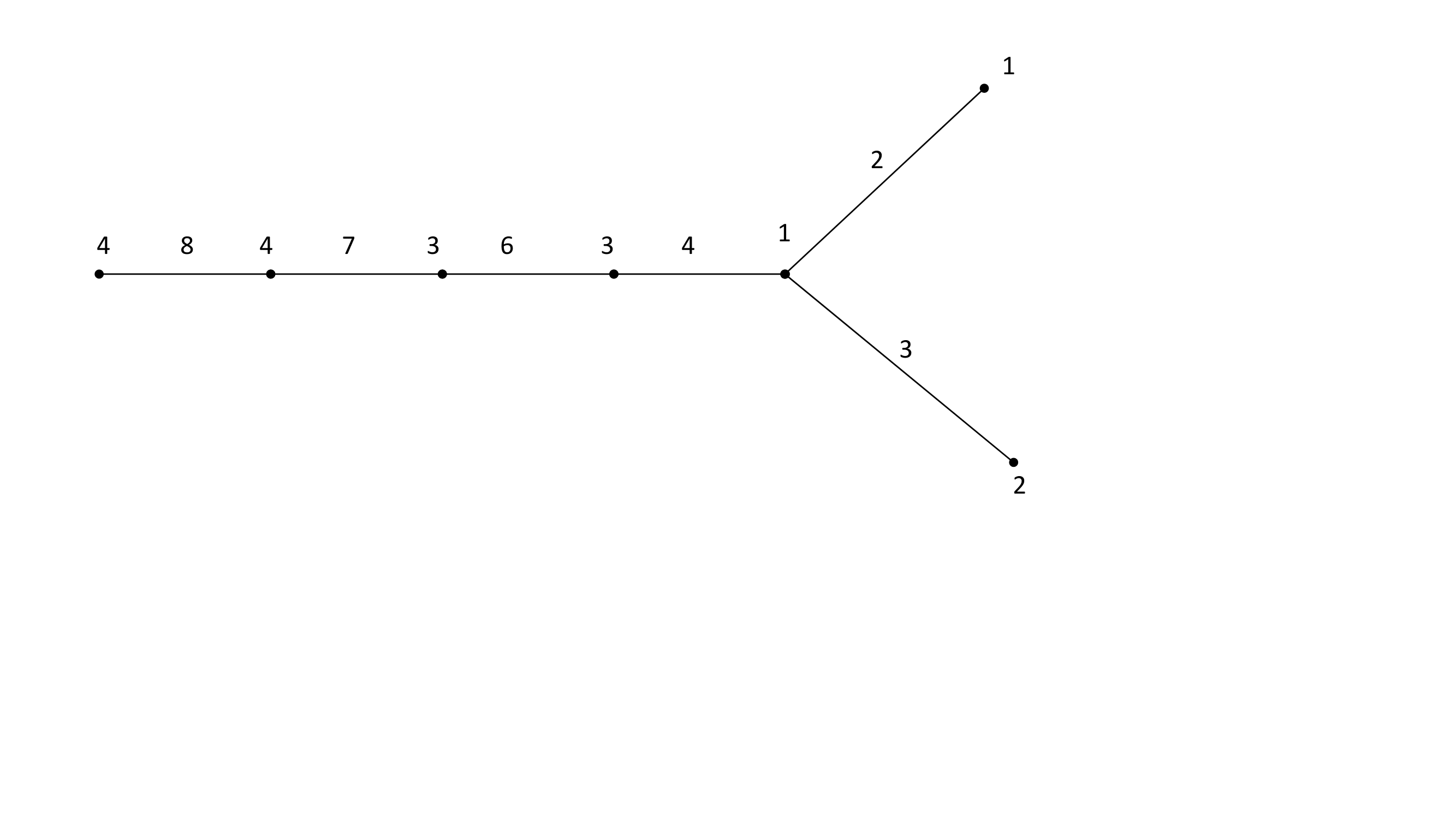}
\end{center}
\begin{center}
Figure 4
\end{center}
\section{Conclusion} 
In this note, we have found the exact value of edge irregularity strength of an important class of graph, namely, Dandelion graph $D(n,l)$, when $\Delta(G) \geq \lceil \frac{|E(G)|+1}{2} \rceil$; and determine the bounds when $ \Delta(G) < \lceil \frac{|E(G)|+1}{2} \rceil $. However, to find the exact value of edge irregularity strength of many other important class of graphs still remain open.

\makeatletter
\renewcommand{\@biblabel}[1]{[#1]\hfill}
\makeatother

\begin{thebibliography}{99}
\bibitem{1} Ahmad, A., Al-Mushayt, O., \& Baca, M. (2014). On edge irregularity strength of graphs. \textit{Applied Mathematics and Computation}, 243, 607--610.

\bibitem{2} Ahmad, A., Baca, M., \& Nadeem, M. F. (2016). On the edge irregularity strength of Toeplitz graphs. \textit{Scientific Bulletin-University Politehnica of Bucharest}, 78, 155--162.

\bibitem{3} Baca, M., Jendrol, S., Miller, M., \& Ryan, J. (2007). On irregular total labellings. \textit{Discrete Mathematics}, 307, 1378--1388.

\bibitem{4} Chartrand, G., Jacobson, M. S., Lehel, J., Oellermann, O. R., \& Saba, F. (1988). Irregular networks. \textit{Congressus Numerantium}, 64, 187--192.

\bibitem{5} Matja\v{z} Krnca \& Riste \v{S}krekovski. (2016). On Wiener Inverse Interval Problem. \emph{MATCH Commun. Math. Comput. Chem}, 75, 71-80.

\bibitem{6} Nagesh, H, M., \& Girish, V. R. (2022). On edge irregularity strength of line graph and line cut-vertex graph of comb graph. \emph{Notes on Number Theory and Discrete Mathematics}, 28(3), 517-524.

\bibitem{7} Tarawneh, I., Hasni, R., \& Ahmad, A. (2016). On the edge irregularity strength of corona product of graphs with paths. \textit{Applied Mathematics E-Notes}, 16, 80--87.

\bibitem{8} Tarawneh, I., Hasni, R., Asim, M. A., \& Siddiqui, M. A. (2019). On the edge irregularity strength of disjoint union of graphs. \textit{Ars Combinatoria}, 142, 239--249.
\end{thebibliography}

\end{document}